\documentclass[11pt,reqno]{amsart}

\usepackage{amsmath,amssymb,mathtools}
\usepackage{microtype}
\usepackage[margin=1.05in]{geometry}
\usepackage{xcolor}
\usepackage{enumitem}
\usepackage{booktabs}
\usepackage{graphicx}
\usepackage{hyperref}
\usepackage[backend=biber,style=numeric,sorting=nyt,maxbibnames=99]{biblatex}

\hypersetup{
  colorlinks=true,
  linkcolor=blue!45!black,
  citecolor=green!35!black,
  urlcolor=blue!55!black,
  pdftitle={Order-Sensitive Fast-Synapse Limits in Sparse Excitatory--Inhibitory Threshold--Reset Networks},
  pdfauthor={Tonic Song},
  pdfsubject={Mathematical preprint},
  pdfkeywords={integrate-and-fire network, sparse mean-field limit, singular synaptic kernel, threshold reset}
}

\numberwithin{equation}{section}

\newtheorem{theorem}{Theorem}[section]
\newtheorem{lemma}[theorem]{Lemma}
\newtheorem{proposition}[theorem]{Proposition}
\newtheorem{corollary}[theorem]{Corollary}
\theoremstyle{definition}
\newtheorem{definition}[theorem]{Definition}

\theoremstyle{remark}
\newtheorem{remark}[theorem]{Remark}

\newcommand{\E}{\mathrm E}
\newcommand{\I}{\mathrm I}
\newcommand{\Prob}{\mathbb P}
\newcommand{\Exp}{\mathbb E}

\newcommand{\one}{\mathbf 1}
\newcommand{\BL}{\mathrm{BL}}
\newcommand{\Lip}{\operatorname{Lip}}
\newcommand{\supp}{\operatorname{supp}}

\title[Order-Sensitive Fast-Synapse Limits]{Order-Sensitive Fast-Synapse Limits in Sparse Excitatory--Inhibitory Threshold--Reset Networks}
\author{Tonic Song}
\date{13 August 2026}

\subjclass[2020]{Primary 60F17, 60K35; Secondary 92B20, 05C80, 34A38}
\keywords{integrate-and-fire network, sparse mean-field limit, singular synaptic kernel, excitation and inhibition, threshold reset, regularization dependence}

\begin{document}

\begin{abstract}
Componentwise weak convergence of signed synaptic kernels does not, by
itself, determine the fast-synapse limit of a sparse threshold--reset
network.  Within a causal event protocol with clamped refractoriness and
smooth positive-delay kernels, we construct two families whose excitatory
and inhibitory measures converge weakly to $\delta_0$ while their microscopic
arrival orders are reversed.  A target fires in the excitatory-first family
and not in the inhibitory-first family precisely when
$x+a-b<\theta\le x+a$.  Strict margins preserve this response under
perturbations of the target state, aggregate E/I pulse masses, and bounded
drift.
The macroscopic effect persists on a moderately sparse Dale-compatible
random block graph with $q_N\to\infty$ and $q_N/N\to0$.  The two systems
share their graph and initial data.  Along every deterministic joint scale
$\varepsilon_N\downarrow0$, their population-averaged firing counts differ by
$1/2+o_{L^1}(1)$.  A bounded-degree construction and a later probe
show that the discrepancy is macroscopic and can persist through reset.
Fixed positive-delay kernels with finitely many classes admit a stable
regime.  Before grazing, typewise-mixing sparse networks converge to a delayed
class mean-field system.  Directed
Erd\H{o}s--R\'enyi graphs yield the bound
$O_{\mathbb P}(\lambda_N^{-1/2}+\|\pi^N-\pi\|_1)$ when
$\lambda_N\to\infty$ and $\lambda_N/N\to0$.  This separates stable averaging
at a fixed delay from singular collapse.  In the latter, componentwise weak
convergence discards signed arrival-order information needed by the
threshold--reset response.
\end{abstract}

\maketitle
\pagestyle{plain}

\section{Introduction}
\label{sec:introduction}

Large interacting-neuron systems are often reduced by sending the network
size to infinity, by replacing sparse connectivity with a population field,
or by idealizing fast synapses as instantaneous impulses.  These operations
are individually familiar, but their combination is delicate in a
threshold--reset system.  A weak topology on input measures records the
total excitatory and inhibitory mass delivered at a limiting time, but not
how those signed masses are traversed inside a collapsing time window.
Threshold crossing and reset can make that missing order visible.

We isolate this selection problem in deterministic threshold--reset networks
on directed signed graphs.  Every emitted spike travels through a smooth
causal kernel with strictly positive delay, so each finite regularized network
has unambiguous event-driven dynamics.  The two regularizations reverse the
order of one excitatory and one inhibitory pulse.  In both cases, the
excitatory and inhibitory kernel measures converge componentwise and weakly to
$\delta_0$, yet the population firing counts can have different limits.

\subsection{Main results}

The causal finite-network protocol is globally well posed and has a
deterministic no-accumulation firing bound; see
Theorem~\ref{thm:finite-wellposedness-paper}.

The local mechanism is elementary but discontinuous.  A node initially at
$x<\theta$, receiving total excitatory input $a$ and total inhibitory input
$-b$, fires only in the excitatory-first ordering exactly when
\begin{equation}
  x+a-b<\theta\le x+a.
  \label{eq:intro-order-window}
\end{equation}
Strict inequalities give a perturbation margin.
Here $F_N^{\varepsilon,o}(t)$ denotes the population-averaged firing count in
\eqref{eq:population-count}.
Theorem~\ref{thm:random-block} lifts this local mechanism to a Dale-compatible
random block graph with vanishing edge weights and fluctuating degrees.  When
$q_N\to\infty$ and $q_N/N\to0$, both regularizations use the same graph and
initial data, and
\[
 \Exp_G\left|
 F_N^{\varepsilon_N,\E}(t_*)
 -F_N^{\varepsilon_N,\I}(t_*)-\frac12
 \right|\longrightarrow0
\]
for every deterministic sequence $\varepsilon_N\downarrow0$.  The graph is
moderately sparse, and its nonzero weights are of order $1/q_N$.
Proposition~\ref{prop:bounded-degree} gives a bounded-degree construction.
In Proposition~\ref{prop:repeated-reset}, a later probe amplifies the
discrepancy after reset.

At fixed smooth kernels with a positive minimum delay, the causal history on
a short time slab is already determined by the preceding slab.  Under
finite-class and transversal-crossing assumptions, this yields a comparison
theorem before grazing.  If the normalized adjacency matrix is typewise
mixing, Theorem~\ref{thm:fixed-kernel-comparison} bounds the mean hybrid
trajectory error by the graph-mixing and class-proportion defects.  For a
directed Erd\H{o}s--R\'enyi graph of expected degree $\lambda_N$,
Corollary~\ref{cor:fixed-kernel-er} gives an
$O_{\Prob}(\lambda_N^{-1/2}+\|\pi^N-\pi\|_1)$ rate when
$\lambda_N\to\infty$ and $\lambda_N/N\to0$.  The constant depends on the
fixed kernels, their positive delay, and the margin to grazing.  The estimate
is not claimed to be uniform in the singular fast-synapse limit.

\subsection{Relation to prior work}

Classical propagation-of-chaos and mean-field theories replace many weak
interactions by an effective law or field
\cite{sznitman1991,ethierkurtz1986}.  Neural versions include sparse balanced
E/I networks \cite{brunel2000}, McKean--Vlasov integrate-and-fire equations
\cite{delarue2015,delarue2015singular}, exact reductions in special quadratic families
\cite{montbrio2015}, and signed Hawkes limits
\cite{pfaffelhuber2022}.  Random-graph limits have been developed through
both diverging-degree averaging and local weak convergence
\cite{delattre2016,lacker2023}.  Recent results reach sparse stochastic
integrate-and-fire networks with singular Poisson jumps
\cite{jabinzhou2026}, arbitrary non-exchangeable $O(1/N)$ weights
\cite{jabinschmutzzhou2024}, and gradual transmission with common noise
\cite{hambly2026}.  These are positive limit theories in regimes with
additional diffusion, intensity regularization, dense scaling, or gradual
one-sign transmission.  Their hypotheses and limiting variables do not
encode the reversed signed within-window ordering studied here.

Smooth feedback kernels can converge to singular hitting-time
McKean--Vlasov dynamics
\cite{inglistalay2015,hambly2025}.  Singular hitting-time particle systems
also exhibit macroscopic jumps, network heterogeneity, and fragility
\cite{nadtochiyshkolnikov2019,nadtochiyshkolnikov2020}.  In sparse graph
cascades, nonnegative regularized impacts
can converge to a minimal physical instantaneous solution
\cite{guoyan2026}.  Those selection mechanisms exploit a one-sign monotonicity
that is absent when excitation, inhibition, and reset interact.  Work on
physical blowups and vanishing buffering or delay likewise shows that a
specified one-sign regularization can select a physically meaningful
solution \cite{sadun2022,papadopoulos2025}.  Those results do not contradict
ours; they concern a specified one-sign regularization, whereas our boundary
concerns signed nonmonotone input.

Finite synaptic amplitudes and pulse shapes are already known to produce
nonlinear, asymmetric, and sometimes instantaneous neural responses
\cite{helias2010,afifurrahman2021,pietras2024}.  We therefore do not claim
that pulse width or timing matters for the first time.  Our main constructions
are a bounded-degree family and a feedforward, moderately sparse random-block
Dale ensemble.  In both, the componentwise weak kernel limits agree, while
the population firing-count separation holds along every deterministic
diagonal sequence $\varepsilon_N\downarrow0$.
Related crossing-process convergence results
require convergence hypotheses strong enough to control the paths and their
first passages \cite{tamborrino2014}; componentwise weak convergence of the
input kernels is deliberately weaker.

Vector-valued impulsive control separates commutative cases from
noncommutative cases in which graph completion or an internal jump path is
part of the solution concept \cite{aronna2015,fusco2024}.  The present result
provides a concrete sparse-network analogue of this order deficit, but not a
formal embedding theorem for impulsive control.  In this analogy, the
rescaled E/I profile inside the collapsing synaptic window plays the role of
missing graph-completion data.  The hard threshold--reset map amplifies that
deficit to a macroscopic observable.

\subsection{Scope of the results}

We do not define atomic $\delta_0$ dynamics, which would require a collision
resolver.  The results do not prove that general iterated limits fail to
commute, and they leave open a canonical instantaneous resolver supplied with
additional ordered data.

The fixed-kernel theorem neither covers grazing nor gives a bound uniform as
the minimum delay vanishes.  All results concern the mathematical model; no
biological validation is claimed.

\section{Delayed threshold--reset networks}
\label{sec:model}

A deterministic finite network already contains the selection obstruction.
Fix a directed graph on vertices
$\{1,\dots,N\}$.  Vertex $j$ has a source type
$\tau_j\in\{\E,\I\}$, threshold $\theta_j$, reset level
$v_{R,j}<\theta_j$, and absolute refractory period
$\rho_j\ge\rho_*>0$.  A weight $w_{ij}$ is associated with the edge
$j\to i$.  Dale's law means
\begin{equation}
  w_{ij}\ge0\quad\text{if }\tau_j=\E,
  \qquad
  w_{ij}\le0\quad\text{if }\tau_j=\I.
  \label{eq:dale}
\end{equation}

For a regularization label $o$ and width parameter $\varepsilon>0$, let
\begin{equation}
 k_\tau^{\varepsilon,o}\in C_c^1((0,\infty)),
 \qquad k_\tau^{\varepsilon,o}\ge0,
 \qquad \int_0^\infty k_\tau^{\varepsilon,o}(s)\,ds=1.
 \label{eq:kernel-assumptions}
\end{equation}
Thus every kernel is causal and vanishes in a neighborhood of zero.  The
right-continuous firing count of vertex $i$ is denoted by
$M_i^{N,\varepsilon,o}$.  Its synaptic input current is
\begin{equation}
 I_i^{N,\varepsilon,o}(t)
 :=\sum_{j=1}^N w_{ij}
 \int_{[0,t)} k_{\tau_j}^{\varepsilon,o}(t-s)
 \,dM_j^{N,\varepsilon,o}(s).
 \label{eq:finite-current}
\end{equation}
The weight is the integrated voltage increment delivered by one spike; the
kernel has units of inverse time.

Let $X_i$ be the voltage and $A_i$ the elapsed time since the most recent
firing.  Outside refractory intervals,
\begin{equation}
 \dot X_i(t)=b_i(X_i(t),A_i(t))+I_i^{N,\varepsilon,o}(t),
 \qquad \dot A_i(t)=1.
 \label{eq:active-dynamics}
\end{equation}
For every finite age horizon, $b_i$ is continuous in both variables,
globally Lipschitz in voltage uniformly over that age range, and has at most
linear growth.  There are no pre-zero spikes or in-flight inputs.  Initially
all vertices are active,
\begin{equation}
 X_i(0-)=x_i^0,
 \qquad A_i(0-)=a_i^0\ge\rho_i,
 \qquad M_i(0-)=0.
 \label{eq:initial-data}
\end{equation}

Every node with $x_i^0\ge\theta_i$ fires once at time zero, and all such
initial firings are processed simultaneously.  When an active voltage first
reaches threshold at a later time $T$, its firing count increases by one,
the voltage is reset to $v_{R,i}$, and the node is clamped at that value on
$[T,T+\rho_i)$.  Synaptic currents continue to evolve during the clamp but
are not integrated by the clamped node.  Upon release, any remaining kernel
tail acts normally.  If several active nodes first reach threshold at the
same time, they fire and reset simultaneously.  The positive delay in
\eqref{eq:kernel-assumptions} prevents a newly emitted spike from causing a
second event at the same mathematical instant.

\begin{theorem}[Finite-network well-posedness]
\label{thm:finite-wellposedness-paper}
For every finite graph, every $\varepsilon>0$, and every regularization
label $o$, the event protocol has a unique global trajectory.  Moreover, for
every $T<\infty$,
\begin{equation}
 M_i^{N,\varepsilon,o}(T)
 \le 1+\left\lfloor\frac{T}{\rho_i}\right\rfloor
 \le 1+\left\lfloor\frac{T}{\rho_*}\right\rfloor.
 \label{eq:firing-bound}
\end{equation}
In particular, bounded intervals contain only finitely many firing events.
\end{theorem}

\begin{proof}
Process the initially superthreshold nodes at time zero.  After finitely
many events, the current generated by the known history is a finite sum of
translated $C_c^1$ kernels and is continuous.  On any interval with fixed
active/refractory status, active voltages therefore solve ordinary
differential equations with continuous time dependence and locally
Lipschitz voltage dependence.  They have unique solutions up to the next
threshold hit, scheduled release, or infinity.  At a threshold event reset
all simultaneous hitters and append their spikes; at a release restart the
corresponding active branch.

Finite-time event accumulation is impossible.  Consecutive firings of node
$i$ are separated by at least $\rho_i$, so the first bound in
\eqref{eq:firing-bound} holds.  With finitely many nodes, infinitely many
firings in a bounded interval would force one node to violate this bound.
Every firing schedules one release, so release times cannot accumulate
either.  The event construction extends globally and is unique.
\end{proof}

\begin{remark}[Why the atomic model is not defined here]
\label{rem:atomic-not-defined}
Theorem~\ref{thm:finite-wellposedness-paper} covers every smooth delayed
kernel used below.  It does not prescribe how an atomic signed input should
be traversed when multiple masses arrive at one time.  Such a prescription
is an extra collision resolver, analogous to graph-completion data in
impulsive control; its absence is tested in
Sections~\ref{sec:local-order}--\ref{sec:macroscopic}.
\end{remark}

For comparing macroscopic outputs, define the population-averaged firing
count
\begin{equation}
 F_N^{\varepsilon,o}(t)
 :=\frac1N\sum_{i=1}^N M_i^{N,\varepsilon,o}(t).
 \label{eq:population-count}
\end{equation}
We call a graph sequence \emph{moderately sparse} when its realized average
degree $d_N:=|E_N|/N$ satisfies $d_N\to\infty$ and $d_N/N\to0$ in the mode
of convergence stated for that sequence.

\section{Weakly identical kernels with different local outputs}
\label{sec:local-order}

Fix a nonnegative $k\in C_c^\infty((0,1))$ with
$\int_0^1k(u)\,du=1$.  For $\varepsilon>0$, define the excitatory-first
family by
\begin{equation}
 k_\E^{\varepsilon,\E}(t)
 =\frac1\varepsilon k\!\left(\frac{t-\varepsilon}{\varepsilon}\right),
 \qquad
 k_\I^{\varepsilon,\E}(t)
 =\frac1\varepsilon k\!\left(\frac{t-3\varepsilon}{\varepsilon}\right),
 \label{eq:e-first-kernels}
\end{equation}
and define the inhibitory-first family by interchanging the two supports:
\begin{equation}
 k_\I^{\varepsilon,\I}(t)
 =\frac1\varepsilon k\!\left(\frac{t-\varepsilon}{\varepsilon}\right),
 \qquad
 k_\E^{\varepsilon,\I}(t)
 =\frac1\varepsilon k\!\left(\frac{t-3\varepsilon}{\varepsilon}\right).
 \label{eq:i-first-kernels}
\end{equation}
For each $\tau\in\{\E,\I\}$ and $o\in\{\E,\I\}$,
\begin{equation}
 k_\tau^{\varepsilon,o}(t)\,dt
 \Longrightarrow \delta_0
 \qquad(\varepsilon\downarrow0).
 \label{eq:componentwise-weak-limit}
\end{equation}
The type-marked weak limit is therefore identical in the two families.  On
the rescaled time $t=\varepsilon u$, however, they retain opposite E/I
orders.

Consider one target starting from $x<\theta$.  A source spike at time zero
delivers total excitatory input $a>0$ and total inhibitory input $-b<0$.
For the moment the target has zero continuous drift and no noise on
$(0,4\varepsilon)$.  It has reset level $v_R<\theta$ and refractory period
$\rho>4\varepsilon$, so it can fire at most once in the microscopic window.

\begin{lemma}[Single-target order criterion]
\label{lem:single-target-order-paper}
The excitatory-first and inhibitory-first families produce different firing
counts on $[0,4\varepsilon]$ if and only if
\begin{equation}
 \boxed{x+a-b<\theta\le x+a.}
 \label{eq:order-criterion}
\end{equation}
In that case the excitatory-first target fires once and the
inhibitory-first target does not fire.
\end{lemma}

\begin{proof}
Before reset, the excitatory-first free path increases from $x$ to $x+a$
and therefore reaches threshold if and only if $x+a\ge\theta$.  The
inhibitory-first free path first decreases from $x$ to $x-b$ and then
increases to $x-b+a$, so it reaches threshold if and only if
$x-b+a\ge\theta$.  The two indicators differ exactly when the first
condition holds and the second fails.  Refractoriness excludes a second
firing.
\end{proof}

The strict margins in \eqref{eq:order-criterion} tolerate small errors in the
state and pulse masses.  This is the form needed in
Section~\ref{sec:macroscopic}, where aggregate random inputs only concentrate
near their nominal values.

\begin{lemma}[Robust separated batch]
\label{lem:robust-batch-paper}
Suppose that for some $\eta>0$,
\begin{equation}
 \min\{\theta-x,\ x+a-\theta,\ \theta-(x+a-b)\}\ge\eta.
 \label{eq:robust-margin}
\end{equation}
Use the separated supports above, but let the actual initial left voltage
and aggregate input magnitudes be $x'$, $A\ge0$, and $B\ge0$.  Suppose that
under both orderings the continuous drift rate before a first hit has
magnitude at most $L_{\rm dr}$ and that
\begin{equation}
 |x'-x|+|A-a|+|B-b|+4L_{\rm dr}\varepsilon<\eta,
 \qquad 4\varepsilon<\rho.
 \label{eq:robust-perturbation}
\end{equation}
Then the target fires exactly once in the excitatory-first system and does
not fire in the inhibitory-first system on $[0,4\varepsilon]$.
\end{lemma}

\begin{proof}
The margin implies $x'<\theta$.  If the excitatory-first target has not hit
by $2\varepsilon$, then the entire excitatory pulse and none of the
inhibitory pulse has arrived, giving
\[
 X^\E(2\varepsilon)
 \ge x+a-|x'-x|-|A-a|-2L_{\rm dr}\varepsilon>\theta,
\]
a contradiction.  Hence it fires by $2\varepsilon$, and the refractory
interval covers the rest of the batch.

For the inhibitory-first system, a hit by $3\varepsilon$ is impossible
because the cumulative synaptic contribution is nonpositive and
$x+|x'-x|+3L_{\rm dr}\varepsilon<\theta$.  A hit between
$3\varepsilon$ and $4\varepsilon$ is also impossible: the inhibitory pulse
has arrived completely and the excitatory contribution is at most $A$, so
\[
 X^\I(t)
 \le x+a-b+|x'-x|+|A-a|+|B-b|+4L_{\rm dr}\varepsilon
 <\theta.
\]
\end{proof}

The discontinuity is not in \eqref{eq:componentwise-weak-limit}; it is in
the response map at a collision of signed impulses.  The weak limit records
the endpoint $x+a-b$ but not whether the path visited $x+a$ first.

\section{Macroscopic separation under componentwise weak kernel limits}
\label{sec:macroscopic}

The local order mechanism persists on graph sequences.  Our main
construction uses random block connectivity with diverging but sublinear
expected degree and vanishing edge weights.  The E-first and I-first systems
share the graph and initial data; only the type-dependent kernels
\eqref{eq:e-first-kernels}--\eqref{eq:i-first-kernels} differ.

\begin{theorem}[Random moderately sparse separation]
\label{thm:random-block}
Let $N=4m$, and let $q_N>0$ satisfy
\begin{equation}
 q_N\longrightarrow\infty,
 \qquad \frac{q_N}{N}\longrightarrow0.
 \label{eq:q-sparse}
\end{equation}
Partition the vertices into source sets $S_\E,S_\I$ and a target set $T$
with sizes $N/4,N/4,N/2$.  Independently retain every possible edge from
$S_\E$ to $T$ and from $S_\I$ to $T$ with probability
\begin{equation}
 p_N=\frac{2q_N}{N}.
 \label{eq:block-probability}
\end{equation}
This is a valid probability for all sufficiently large $N$ by
\eqref{eq:q-sparse}.  Give vertices in $S_\E$ and $S_\I$ types $\E$ and
$\I$, respectively; target types may be arbitrary.
Give retained excitatory edges weight $2a/q_N$ and retained inhibitory
edges weight $-2b/q_N$, and add no other edges.  Initialize every source at
threshold and every target at the same $x<\theta$.  All drifts vanish.
Assume that for some $\eta>0$,
\begin{equation}
 x+a\ge\theta+3\eta,
 \qquad x+a-b\le\theta-3\eta.
 \label{eq:random-margin}
\end{equation}
Then, for every fixed $t_*>0$ and every deterministic sequence
$\varepsilon_N\downarrow0$,
\begin{equation}
 F_N^{\varepsilon_N,\E}(t_*)
 -F_N^{\varepsilon_N,\I}(t_*)
 \xrightarrow[N\to\infty]{L^1(\Prob_G)}\frac12.
 \label{eq:random-gap}
\end{equation}
Moreover, if $d_N(G)=|E_N|/N$, then
\begin{equation}
 \frac{d_N(G)}{q_N}\xrightarrow{\Prob_G}\frac12.
 \label{eq:degree-concentration}
\end{equation}
If in addition $q_N/\log N\to\infty$, then the count difference equals
$1/2$ with probability tending to one.  The graph obeys Dale's law.
\end{theorem}

\begin{proof}
For a target $i$, let $D_{i,\E}$ and $D_{i,\I}$ be its excitatory and
inhibitory in-degrees and write
\begin{equation}
 A_i=\frac{2a}{q_N}D_{i,\E},
 \qquad B_i=\frac{2b}{q_N}D_{i,\I}.
 \label{eq:aggregate-random-inputs}
\end{equation}
Under the common-graph coupling, $(A_i,B_i)$ is identical in the two
regularizations.  Also
\[
 D_{i,\E},D_{i,\I}
 \sim\operatorname{Bin}\!\left(\frac N4,\frac{2q_N}{N}\right),
 \qquad
 \Exp_GD_{i,\E}=\Exp_GD_{i,\I}=\frac{q_N}{2}.
\]
Standard Chernoff bounds give $c=c(a,b,\eta)>0$ such that, for
\[
 \mathcal G_i:=\{|A_i-a|\le\eta,\ |B_i-b|\le\eta\},
\]
we have
\begin{equation}
 \Prob_G(\mathcal G_i^c)\le4e^{-cq_N}.
 \label{eq:target-chernoff}
\end{equation}

For all large $N$, $4\varepsilon_N<\min\{t_*,\rho_*\}$.  On
$\mathcal G_i$, the strict bounds
$x+A_i\ge\theta+2\eta$ and
$x+A_i-B_i\le\theta-\eta$ give the conclusion directly.  Under E-first,
the target path is monotone increasing until the end of the excitatory
pulse and therefore must cross threshold.  Under I-first, it first moves
downward and then increases monotonically only to the subthreshold terminal
value $x+A_i-B_i$; hence it never fires.  Refractoriness makes the target
fire at most once under either ordering.  Define
\[
 Z_i^N:=M_i^{N,\varepsilon_N,\E}(t_*)
       -M_i^{N,\varepsilon_N,\I}(t_*),
 \qquad
 H_N:=\frac1{|T|}\sum_{i\in T}\one_{\mathcal G_i^c}.
\]
Then $Z_i^N=1$ on $\mathcal G_i$, $|Z_i^N|\le1$, and
\[
 \left|\frac1{|T|}\sum_{i\in T}Z_i^N-1\right|\le2H_N,
 \qquad
 \Exp_GH_N\le4e^{-cq_N}\longrightarrow0.
\]
The source firing counts agree exactly and $|T|/N=1/2$, proving
\eqref{eq:random-gap}.  Notice that no independence among the good-target
events was used and no relative-rate condition on $q_N$ and
$\varepsilon_N$ appears.

There are $N^2/4$ possible edges, so
\[
 |E_N|\sim\operatorname{Bin}\!\left(\frac{N^2}{4},
                                    \frac{2q_N}{N}\right),
 \qquad \Exp_G|E_N|=\frac{Nq_N}{2}.
\]
Another Chernoff estimate yields \eqref{eq:degree-concentration}.  A union
bound in \eqref{eq:target-chernoff} gives
$\Prob_G(\bigcup_{i\in T}\mathcal G_i^c)\le2Ne^{-cq_N}$ after changing
$c$, which proves the exact high-probability statement.  Edge sign is
determined by the source block, establishing Dale's law.
\end{proof}

Diverging degree is not required for the separation.  The following
bounded-degree construction isolates the same local mechanism, while
Theorem~\ref{thm:random-block} gives the large-random-network result.

\begin{proposition}[Bounded-degree separation]
\label{prop:bounded-degree}
There are Dale-compatible signed directed networks whose maximum in- and
out-degrees are uniformly bounded, such that for every fixed $t_*>0$ and
every sequence $\varepsilon_N\downarrow0$,
\begin{equation}
 F_N^{\varepsilon_N,\E}(t_*)
 -F_N^{\varepsilon_N,\I}(t_*)\longrightarrow\frac13.
 \label{eq:bounded-gap}
\end{equation}
\end{proposition}

\begin{proof}
Let $n_N=\lfloor N/3\rfloor$ and form $n_N$ disconnected motifs
\[
 E_\ell\longrightarrow C_\ell\longleftarrow I_\ell,
 \qquad \ell=1,\dots,n_N,
\]
leaving other nodes isolated.  Use weights $+a$ and $-b$, initialize each
source at threshold and each target at $x$, and choose
$x+a-b<\theta<x+a$.  Set all drifts to zero and initialize the isolated
vertices strictly below threshold.  Sources fire at time zero.  By
Lemma~\ref{lem:single-target-order-paper}, each target fires once only in
the E-first system.  Once $4\varepsilon_N<\min\{t_*,\rho_*\}$, the count gap
is $n_N/N\to1/3$.  Targets have in-degree two, sources have out-degree one,
and source types determine edge signs.
\end{proof}

A later macroscopic probe can read the local choice after the first reset.
The discrepancy is therefore not confined to the vanishing window in which
the two pulse batches arrive.

\begin{proposition}[Persistence through reset]
\label{prop:repeated-reset}
Fix $\rho>0$ and $T_P>\rho$.  There are uniformly bounded-degree,
Dale-compatible networks with the same initial data in both orderings such
that, for every $t_*>T_P$ and every $\varepsilon_N\downarrow0$,
\begin{equation}
 F_N^{\varepsilon_N,\E}(\rho/2)
 -F_N^{\varepsilon_N,\I}(\rho/2)\longrightarrow\frac14,
 \qquad
 F_N^{\varepsilon_N,\E}(t_*)
 -F_N^{\varepsilon_N,\I}(t_*)\longrightarrow\frac12.
 \label{eq:repeated-reset-gaps}
\end{equation}
For all sufficiently large $N$ along the stated sequence, every nonisolated
target has fired twice by $t_*$ under E-first and has not fired under
I-first.
\end{proposition}

\begin{proof}
The explicit four-node motif and its verification are given in
Appendix~\ref{app:repeated-reset}.  The first signed batch separates the
states as above.  After refractory release, a common delayed excitatory
probe crosses threshold only from the E-first reset state, creating a
second population-scale count difference.
\end{proof}

\begin{remark}[Data required for selection]
Both regularizations in Theorem~\ref{thm:random-block} have the same
componentwise type-marked weak kernel limit and use the same realized graph.
Equation~\eqref{eq:random-gap} therefore excludes any macroscopic firing
limit determined only by those data on this class.  An instantaneous model
may still be well posed if its state records the rescaled E/I order or an
equivalent collision resolver.
\end{remark}

\section{Exact finite-size anatomy of the nonselection mechanism}
\label{sec:numerics}

The results above admit finite-dimensional calculations that retain the
threshold--reset mechanism instead of replacing it by a continuous surrogate.
We use them for two purposes: to resolve the finite-size transition in
Theorem~\ref{thm:random-block}, and to display the post-reset readout in
Proposition~\ref{prop:repeated-reset}.  Event times are obtained from the
smooth cumulative kernel by root finding; no time grid decides whether a
threshold is crossed.

\subsection{From the local selection region to the random-block limit}

Fix $x=0.2$ and $\theta=1$.  If a target receives aggregate excitatory and
inhibitory masses $(A,B)$, Lemma~\ref{lem:single-target-order-paper} gives the
selection region
\begin{equation}
 A\ge0.8,
 \qquad A-B<0.8.
 \label{eq:numerical-local-selection-region}
\end{equation}
The nominal point $(a,b)=(1.1,0.6)$ lies strictly inside this region.  The
box $|A-a|\le0.1$, $|B-b|\le0.1$ also lies inside, matching the concentration
event used in the proof of Theorem~\ref{thm:random-block}.

For a random-block target, let
\[
 D_\E,D_\I\stackrel{\mathrm{ind}}{\sim}
 \operatorname{Bin}\!\left(\frac N4,\frac{2q}{N}\right).
\]
The exact finite-size separation probability is
\begin{equation}
 r_{N,q}(a,b)
 =\Prob\!\left(
 x+\frac{2a}{q}D_\E\ge\theta,
 \quad
 x+\frac{2a}{q}D_\E-\frac{2b}{q}D_\I<\theta
 \right).
 \label{eq:numerical-random-block-probability}
\end{equation}
Consequently, the population gap has the exact law
\begin{equation}
 N\bigl(F_N^{\varepsilon_N,\E}(t_*)
       -F_N^{\varepsilon_N,\I}(t_*)\bigr)
 \sim\operatorname{Bin}(N/2,r_{N,q}),
 \label{eq:numerical-random-block-law}
\end{equation}
once $4\varepsilon_N<\min\{t_*,\rho_*\}$.  Figure~\ref{fig:random-block-selection}
evaluates this law directly.  The finite random graph smooths both boundaries
of \eqref{eq:numerical-local-selection-region}; increasing $q$ sharpens them.
Along $q_N=\lfloor N^{2/3}\rfloor$ at the nominal point, the exact mean gap
approaches $1/2$ while its graph-to-graph interval collapses.

\begin{figure}[t]
 \centering
 \includegraphics[width=\textwidth]{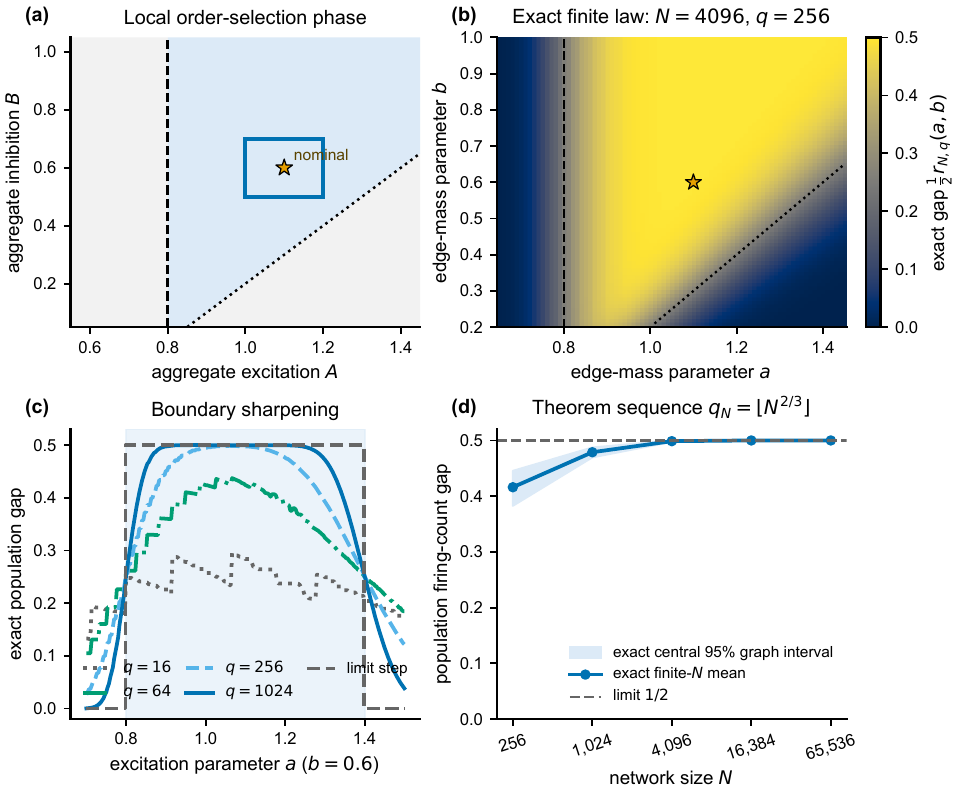}
 \caption{Local and random-block selection.  (a) The exact local region in
 $(A,B)$, with the nominal point and the concentration box used in the random
 proof.  (b) The exact finite-size mean population gap
 $\tfrac12r_{N,q}(a,b)$ for $N=4096$ and $q=256$; black curves are the local
 boundaries.  (c) At $b=0.6$ and $N=65536$, increasing $q$ sharpens the two
 transitions toward the limiting half-height window $0.8\le a<1.4$.
 (d) Exact mean and central $95\%$ interval of the
 $\operatorname{Bin}(N/2,r_{N,q})/N$ graph law along
 $q_N=\lfloor N^{2/3}\rfloor$ at $(a,b)=(1.1,0.6)$; the interval contains no
 Monte Carlo uncertainty.}
 \label{fig:random-block-selection}
\end{figure}

\subsection{A common probe reads the branch after reset}

We next use the four-node motif from Appendix~\ref{app:repeated-reset}.  Set
\[
 \theta=1,
 \quad x=0,
 \quad a=\frac32,
 \quad b=1,
 \quad v_R=\frac34,
 \quad c=\frac38.
\]
The first signed batch leaves the I-first target at
$y=x+a-b=1/2$, whereas the E-first target fires and is later released at
$v_R=3/4$.  A common delayed probe then separates these states because
\begin{equation}
 y+c=\frac78<1<\frac98=v_R+c.
 \label{eq:numerical-probe-margin}
\end{equation}

For the normalized bump used in Section~\ref{sec:local-order}, let $K$ be its
cumulative mass.  With $\varepsilon=0.05$, $\rho=0.5$, probe time $T_P=1$,
and observation time $t_*=1.3$, both E-first crossings satisfy
$K(s_*)=2/3$, where $s_*=0.5327246820$.  They occur at
\[
 t_1=0.0766362341,
 \qquad
 t_2=1.0766362341.
\]
The temporal separation condition is strict:
\[
 4\varepsilon=0.2
 <\min\{\rho/2,T_P-\rho,t_*-T_P\}=0.25.
\]

If $n_N=\lfloor N/4\rfloor$ motifs are used, the population count difference
is therefore exactly
\begin{equation}
 \Delta F_N(t)=
 \begin{cases}
  0,&t<t_1,\\
  n_N/N,&t_1\le t<t_2,\\
  2n_N/N,&t\ge t_2.
 \end{cases}
 \label{eq:numerical-probe-staircase}
\end{equation}
The probe effect is not confined to the nominal parameter point.  Away from
the two threshold-equality boundaries, the final target-count difference is
\begin{equation}
 \Delta M_C
 =1+\one_{\{v_R+c>\theta\}}-\one_{\{y+c>\theta\}}.
 \label{eq:numerical-probe-phase}
\end{equation}
Thus $v_R+c>\theta$ and $y+c<\theta$ define an open region in which the
probe increases the per-target discrepancy from one firing to two.

\begin{figure}[t]
 \centering
 \includegraphics[width=\textwidth]{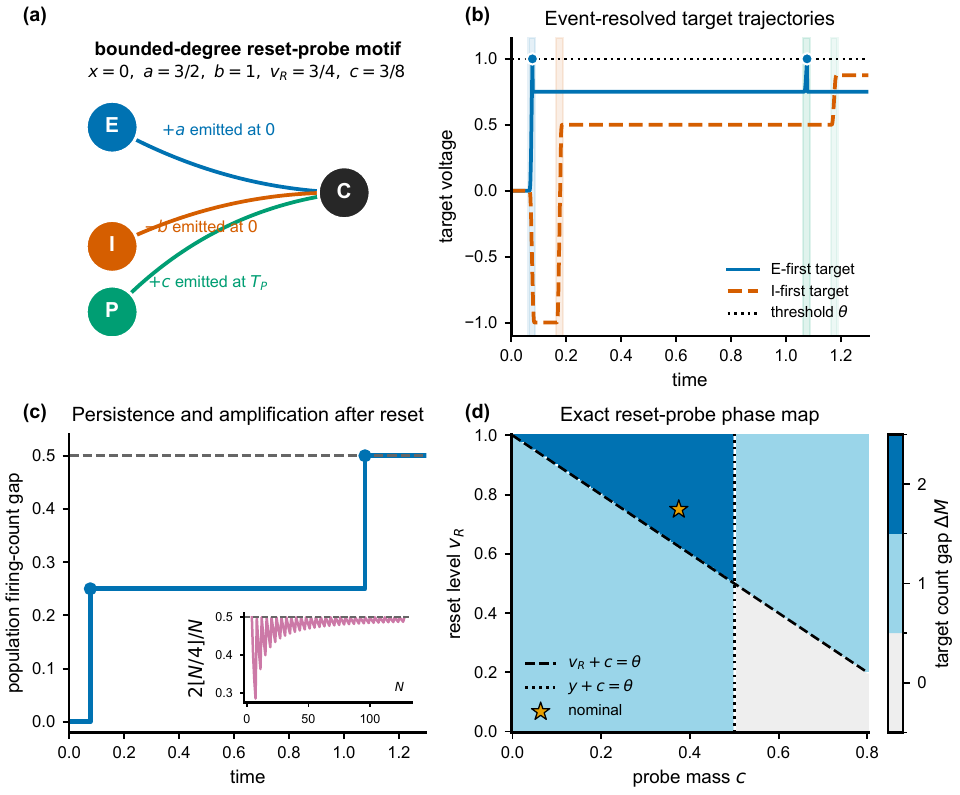}
 \caption{A bounded-degree readout after reset.  (a) The four-node motif;
 $E$ and $I$ deliver the first signed batch, while $P$ delivers the common
 delayed probe.  (b) Event-resolved target voltages: E-first crosses twice,
 while I-first ends below threshold after both batches.  Shading marks pulse
 supports, and threshold events are root-solved independently of the plotting
 mesh.  (c) For the displayed $N=64$ population, the exact gap rises from
 $n_N/N$ after the first batch to $2n_N/N$ after the probe; the inset records
 the finite-$N$ floor effect.
 (d) Exact reset--probe phase diagram in $(c,v_R)$ for $y=1/2$; the nominal
 point lies inside the open two-firing-difference region, and broken black curves
 mark non-robust threshold equalities.}
 \label{fig:reset-probe-persistence}
\end{figure}

The archived workflow records the analytic inputs, root residuals, finite
binomial sums, phase grids, software versions, and SHA-256 hashes.  A clean
verification run regenerates every table and figure and requires byte-level
agreement.  These calculations expose the finite-size structure of the
proved constructions; they do not replace the proofs or supply biological
validation.

\clearpage
\section{A fixed-kernel mean-field comparison before grazing}
\label{sec:fixed-kernel}

The preceding counterexamples let a positive delay collapse with $N$.  Fix
instead one smooth delayed kernel family.  Sparse averaging is stable in the
pre-grazing regime described below; the result does not address the
collapsing-kernel random block setting of Theorem~\ref{thm:random-block}.

\subsection{Finite-class delayed mean field}

Let $\mathcal Q=\{1,\dots,Q\}$ be a finite class set.  A class $q$ fixes a
source type $\tau(q)\in\{\E,\I\}$, drift $b_q$, threshold $\theta_q$,
reset level $v_{R,q}<\theta_q$, refractory period $\rho_q\ge\rho_*$, and
initial data $x_q^0<\theta_q$, $a_q^0\ge\rho_q$.  On every finite age interval,
$b_q$ is jointly Lipschitz in voltage and age and has at most linear voltage
growth.  Set
\begin{equation}
 J_\E=\alpha_\E>0,
 \qquad J_\I=-\alpha_\I<0.
 \label{eq:type-amplitudes}
\end{equation}
Fix nonnegative normalized kernels $k_\tau\in C_c^1((0,\infty))$ with
$\int_0^\infty k_\tau(s)\,ds=1$, and put
\begin{equation}
 \delta_k:=\min_{\tau\in\{\E,\I\}}\inf\supp k_\tau>0.
 \label{eq:positive-delay}
\end{equation}

At size $N$, node $i$ has class $q_i^N$.  Write
\begin{equation}
 \pi_q^N:=\frac1N\#\{i:q_i^N=q\},
 \qquad \pi_q^N\longrightarrow\pi_q>0,
 \qquad \sum_q\pi_q=1.
 \label{eq:class-proportions}
\end{equation}
For a directed adjacency matrix $A_N=(a_{ij}^N)\in\{0,1\}^{N\times N}$ with
zero diagonal and an expected-degree scale $\lambda_N>0$, define
\begin{equation}
 w_{ij}^N=\frac{a_{ij}^N}{\lambda_N}J_{\tau(q_j^N)}.
 \label{eq:fixed-kernel-weights}
\end{equation}
The finite network follows Section~\ref{sec:model} with the fixed kernels
$k_\tau$.

The candidate delayed mean field has one trajectory
$Z_q=(X_q,A_q,M_q)$ per class.  Its emitted trace and common input are
\begin{equation}
 H_q(t):=\int_{[0,t)}k_{\tau(q)}(t-s)\,dM_q(s),
 \qquad
 I(t):=\sum_{q=1}^Q\pi_qJ_{\tau(q)}H_q(t).
 \label{eq:class-mean-field}
\end{equation}
While active, class $q$ solves
$\dot X_q=b_q(X_q,A_q)+I$ and otherwise follows the same reset and clamp
protocol.

\begin{lemma}[Delayed class well-posedness]
\label{lem:class-wellposedness}
The class system has a unique global trajectory and satisfies the firing
bound \eqref{eq:firing-bound} with $\rho_*$.
\end{lemma}

\begin{proof}
Choose $h\in(0,\delta_k)$.  On $[0,h]$, no spike emitted at nonnegative
time can yet contribute, so each class follows an event ODE with known
input.  Once the trajectories are known through $mh$, the input on
$[mh,(m+1)h]$ depends only on spikes before
$(m+1)h-\delta_k<mh$ and is already known.  Iteration and
Theorem~\ref{thm:finite-wellposedness-paper} give global existence,
uniqueness, and the count bound.
\end{proof}

\begin{definition}[Regular pre-grazing horizon]
\label{def:regular-horizon-paper}
A time $T>0$ is regular for the class solution when it is neither a firing
nor a release time of any class and there is $\gamma_T>0$ such that every
class firing time $s\in(0,T)$ is transversal:
\begin{equation}
 b_q(\theta_q,A_q(s-))+I(s)\ge\gamma_T.
 \label{eq:transversality}
\end{equation}
The first grazing time is the first threshold contact at which this lower
bound fails.
\end{definition}

For a signed measure $\nu$ on $[0,T]$, set
\begin{equation}
 \|\nu\|_{\BL,T}^*
 :=\sup\left\{
 \left|\int\varphi\,d\nu\right|:
 \|\varphi\|_\infty+\Lip(\varphi)\le1
 \right\}.
 \label{eq:bl-norm}
\end{equation}
Let $\mathcal R_q[J]=(X_q[J],A_q[J],M_q[J])$ be the clamped response to a
continuous current $J$, and let $H_q[J]=k_{\tau(q)}*dM_q[J]$.  Define
\begin{equation}
\begin{aligned}
 \mathfrak d_{q,T}(J,\widetilde J):={}&
 \int_0^T\bigl(
 |X_q[J]-X_q[\widetilde J]|\wedge1
 +|A_q[J]-A_q[\widetilde J]|\wedge1\bigr)\,dt\\
 &+\|H_q[J]-H_q[\widetilde J]\|_{L^1(0,T)}
 +\|dM_q[J]-dM_q[\widetilde J]\|_{\BL,T}^*.
\end{aligned}
\label{eq:hybrid-metric}
\end{equation}

\begin{lemma}[Transversal response stability]
\label{lem:response-stability-paper}
If $T$ is regular, there is $C_{\rm hit}<\infty$ such that, for every class
$q$ and every continuous current $J$,
\begin{equation}
 \mathfrak d_{q,T}(J,I)
 \le C_{\rm hit}\|J-I\|_{L^1(0,T)}.
 \label{eq:response-stability}
\end{equation}
For sufficiently small current error, the firing counts agree and matched
firing times differ by at most
$C_{\rm hit}\|J-I\|_{L^1(0,T)}$.
\end{lemma}

\begin{proof}
The reference event list is finite.  Away from its firing times, active
reference branches have a positive threshold gap.  Near a firing time,
\eqref{eq:transversality} gives a neighborhood in which the free branch
crosses at speed at least $\gamma_T/2$.  Before the first event, Gronwall's
inequality controls the free-branch error by
$C\|J-I\|_{L^1}$, and the crossing-speed bound controls the displacement of
the first firing time.

Inductively suppose firing times $s_1,\dots,s_r$ have been matched with
$\widetilde s_1,\dots,\widetilde s_r$ and write
$d_\ell=|s_\ell-\widetilde s_\ell|$.  The reset and release times after
firing $r$ are displaced by $d_r$.  Voltage error may be order one on this
phase-mismatch layer, but the layer has length $d_r$; after both paths are
released, the mismatch contribution obeys
\[
 \int_{\text{mismatch}}|J(t)|\,dt
 \le \|J-I\|_{L^1(0,T)}+\|I\|_\infty d_r.
\]
Gronwall then gives a free-branch error bounded by
\[
 C\left(\|J-I\|_{L^1}+\sum_{\ell=1}^r d_\ell\right).
\]
The positive reset gap excludes an extra firing inside the mismatch layer,
and transversality yields the same bound for $d_{r+1}$.  Finite induction
therefore controls all event displacements and the integrated state error.
Finally, the translation estimate
\[
 \|k_\tau(\cdot-s)-k_\tau(\cdot-\widetilde s)\|_{L^1}
 \le\|k_\tau'\|_{L^1}|s-\widetilde s|
\]
controls emitted traces, while matched atoms control the bounded-Lipschitz
firing measure.  Large current errors are absorbed by the uniform firing
bound after enlarging the constant.
\end{proof}

\subsection{A mixing estimate uniform over graph-dependent histories}

For a real matrix $B$, let
$\|B\|_{\infty\to1}:=\sup_{\|u\|_\infty\le1}\|Bu\|_1$ and let $P_\tau^N$
project onto nodes of source type $\tau$.  Define
\begin{equation}
 \Delta_N:=
 \max_{\tau\in\{\E,\I\}}\frac1N
 \left\|
 \left(\frac{A_N}{\lambda_N}-\frac{\one\one^\top}{N}\right)P_\tau^N
 \right\|_{\infty\to1}.
 \label{eq:mixing-defect}
\end{equation}
The supremum is taken over all vectors and therefore applies even when the
firing history is graph-dependent.

\begin{lemma}[Graph-dependent input bound]
\label{lem:graph-input-paper}
Let $M_1,\dots,M_N$ be arbitrary firing counts satisfying
\eqref{eq:firing-bound} through $T$, and set
$H_j=k_{\tau(q_j^N)}*dM_j$.  Then
\begin{equation}
 \frac1N\sum_{i=1}^N\int_0^T
 \left|
 \sum_{j=1}^N J_{\tau(q_j^N)}
 \left(\frac{a_{ij}^N}{\lambda_N}-\frac1N\right)H_j(t)
 \right|dt
 \le C_{T,\mathrm{data}}\Delta_N,
 \label{eq:graph-input-bound}
\end{equation}
where the constant may depend on the fixed class, kernel, coupling, and
refractory data, but is independent of the graph and firing histories.
\end{lemma}

\begin{proof}
The refractory bound gives
$0\le H_j(t)\le K_T\max_\tau\|k_\tau\|_\infty$, where
$K_T=1+\lfloor T/\rho_*\rfloor$.  At each time, apply
\eqref{eq:mixing-defect} separately to the two bounded vectors
$P_\tau^N(H_j(t))_j$, then integrate and sum over $\tau$.
\end{proof}

\subsection{Comparison and the Erd\H{o}s--R\'enyi consequence}

Let $I_i^N$ denote the current actually received by node $i$ and let $I$ be
the class mean-field current in \eqref{eq:class-mean-field}.

\begin{theorem}[Fixed-kernel comparison before grazing]
\label{thm:fixed-kernel-comparison}
Fix the kernels and class data above and let $T$ be a regular horizon.  There
is $C_{T,\mathrm{data}}<\infty$ such that, for all sufficiently large $N$,
\begin{equation}
 \frac1N\sum_{i=1}^N\left[
 \mathfrak d_{q_i^N,T}(I_i^N,I)
 +\|I_i^N-I\|_{L^1(0,T)}
 \right]
 \le C_{T,\mathrm{data}}\left(
 \Delta_N+\sum_{q=1}^Q|\pi_q^N-\pi_q|
 \right).
 \label{eq:fixed-kernel-comparison}
\end{equation}
No smallness condition on $\alpha_\E+\alpha_\I$ is required.
\end{theorem}

\begin{proof}
Choose a partition $0=t_0<\cdots<t_L=T$ whose mesh is smaller than
$\delta_k$ and whose interior endpoints avoid class firing and release
times.  Let
\[
 D_m:=\frac1N\sum_{i=1}^N
 \mathfrak d_{q_i^N,t_m}(I_i^N,I).
\]
Decompose $I_i^N-I$ into the centered adjacency term, an average difference
between sparse and class emitted traces, and the class-proportion error.
Lemma~\ref{lem:graph-input-paper} controls the first term by $C\Delta_N$;
the last is at most $C\sum_q|\pi_q^N-\pi_q|$.

For $t\le t_m$, an emitted trace only uses spikes at
$s\le t-\delta_k<t_{m-1}$.  The bounded-Lipschitz part of the hybrid metric,
applied to the test function $s\mapsto k_\tau(t-s)$ and integrated in $t$,
therefore controls the middle term by $CD_{m-1}$.  Hence
\[
 \frac1N\sum_i\|I_i^N-I\|_{L^1(0,t_m)}
 \le C\left(\Delta_N+\sum_q|\pi_q^N-\pi_q|+D_{m-1}\right).
\]
Lemma~\ref{lem:response-stability-paper}, applied nodewise and averaged,
gives the same recursion for $D_m$.  Since $D_0=0$ and $L$ is finite,
induction proves \eqref{eq:fixed-kernel-comparison}.  The argument is causal
rather than contractive: newly displaced spikes cannot affect the same slab
because of the positive delay.
\end{proof}

\begin{lemma}[Typewise Erd\H{o}s--R\'enyi mixing]
\label{lem:er-mixing-paper}
Let $1\le\lambda_N\le N$, let $a_{ii}^N=0$, and let $a_{ij}^N$, $i\ne j$,
be independent $\operatorname{Bernoulli}(\lambda_N/N)$ variables.  For deterministic class
assignments, or after conditioning on assignments independent of the graph,
there are numerical $C,c>0$ such that
\begin{equation}
 \Prob_G\left(
 \Delta_N>\frac{C}{\sqrt{\lambda_N}}+\frac1N
 \right)\le4e^{-cN}.
 \label{eq:er-mixing}
\end{equation}
\end{lemma}

\begin{proof}
After centering $A_N$, the extreme points of the
$\ell^\infty\!\to\!\ell^1$ norm are sign vectors.  For fixed
$\tau$ and $u,v\in\{-1,1\}^N$, Bernstein's inequality bounds
$v^\top(A_N-\Exp A_N)P_\tau^Nu$, whose variance is at most
$N\lambda_N$, by $LN\sqrt{\lambda_N}$ with probability
$1-2e^{-c_LN}$.  A union bound over two types and at most $4^N$ sign-vector
pairs gives the centered part of \eqref{eq:er-mixing}; the omitted diagonal
contributes $1/N$.
\end{proof}

\begin{corollary}[Moderately sparse fixed-kernel limit]
\label{cor:fixed-kernel-er}
Assume
$\lambda_N\to\infty$, $\lambda_N/N\to0$, and $\pi^N\to\pi$.  On the
directed Erd\H{o}s--R\'enyi graph of
Lemma~\ref{lem:er-mixing-paper}, at every regular horizon,
\begin{equation}
 \frac1N\sum_{i=1}^N
 \mathfrak d_{q_i^N,T}(I_i^N,I)
 =O_{\Prob_G}\!\left(
 \frac1{\sqrt{\lambda_N}}+\sum_q|\pi_q^N-\pi_q|
 \right).
 \label{eq:fixed-kernel-er-rate}
\end{equation}
Writing
\[
 \Sigma_N:=\frac1N\sum_{i=1}^N
 \delta_{\tau(q_i^N)}\otimes dM_i^N,
 \qquad
 \Sigma:=\sum_{q=1}^Q\pi_q\,
 \delta_{\tau(q)}\otimes dM_q,
\]
the measures $\Sigma_N$ converge in probability to $\Sigma$ in
bounded-Lipschitz distance on $\{\E,\I\}\times[0,T]$, with the discrete
metric on the type coordinate.
The realized average degree satisfies $d_N/\lambda_N\to1$ in probability.
\end{corollary}

\begin{proof}
Combine Theorem~\ref{thm:fixed-kernel-comparison} with
Lemma~\ref{lem:er-mixing-paper}.  The firing-measure component of the hybrid
metric, together with the class-proportion error in
\eqref{eq:fixed-kernel-comparison}, gives the measure assertion.  Finally,
$|E_N|\sim\operatorname{Bin}(N(N-1),\lambda_N/N)$, and Chebyshev's
inequality gives the degree claim.
\end{proof}

\begin{remark}[Dependence on the delay scale]
\label{rem:nonuniformity}
The constant in \eqref{eq:fixed-kernel-comparison} depends on the minimum
delay, kernel norms and derivatives, the event list, and the margin to
grazing.  It may deteriorate as the kernel collapses.  Consequently,
Corollary~\ref{cor:fixed-kernel-er} does not pass to $\delta_0$, and no bound
uniform over the collapsing-kernel scale is asserted.  This is consistent
with Theorem~\ref{thm:random-block}.
\end{remark}

\section{Selection boundary and implications}
\label{sec:discussion}

At a fixed positive delay and before grazing, the response map is stable in
an $L^1$ current metric, and sparse graph mixing propagates causally from one
time slab to the next.  Under singular collapse, componentwise weak
convergence forgets the path inside the shrinking E/I batch.  Threshold and
reset convert the missing path information into a discrete event; graph
replication makes the difference macroscopic.

The negative theorem is a diagonal statement.  For every deterministic
$\varepsilon_N\downarrow0$, two regularized families have the same weak kernel
data but different firing-count limits; the theorem neither compares two
general iterated limits nor defines an atomic network against which either
family must converge.

The two graph ensembles have different roles.  The fixed-kernel result
assumes typewise mixing of the normalized
adjacency matrix and controls a finite-class delayed mean field.  The random
block counterexample concentrates the aggregate input received by each
target, but its block structure is not asserted to satisfy the global
mixing hypothesis of Theorem~\ref{thm:fixed-kernel-comparison}.  The
fixed-kernel theorem thus describes a safe regime rather than a converse on
the same ensemble.

A singular theory that assigns a unique firing-count limit on this class must
retain additional structure or exclude the order-sensitive configurations.
One may retain the rescaled type-marked timing profile as a
graph-completion variable, prescribe an atomic collision resolver, impose a
commutative response rule, or prove that the limiting population places no
mass in the order-sensitive threshold set.  Any such proposal must also
supply estimates uniform as the delay vanishes.  Weak convergence of the
two component kernel measures, by itself, supplies none of these data.

Replacing narrow kernels by their masses alone need not preserve population
spike counts in a signed threshold--reset network.  If excitatory and
inhibitory contributions collide near threshold, their within-step or
within-pulse order can select different counts even when all
bounded-continuous typewise tests of the kernel measures have the same limit.
An instantaneous signed threshold--reset model in this class therefore needs
a specified collision convention.  This does not rule out instantaneous
models.  The convention becomes part of both the model specification and any
convergence claim.

\appendix
\section{Explicit persistence-through-reset construction}
\label{app:repeated-reset}

We prove Proposition~\ref{prop:repeated-reset}.  Choose $a,b,c>0$,
$x<\theta$, and $v_R<\theta$ so that
\begin{equation}
 x+a-b+c<\theta<\min\{x+a,v_R+c\}.
 \label{eq:probe-parameters}
\end{equation}
For example, take
\[
 \theta=1,
 \quad x=0,
 \quad a=\frac32,
 \quad b=1,
 \quad v_R=\frac34,
 \quad c=\frac38.
\]

For each $N$, let $n_N=\lfloor N/4\rfloor$ and form $n_N$ disconnected
four-node motifs
\begin{equation}
 E_\ell\longrightarrow C_\ell\longleftarrow I_\ell,
 \qquad
 P_\ell\longrightarrow C_\ell,
 \qquad \ell=1,\dots,n_N,
 \label{eq:probe-motif}
\end{equation}
with weights $+a,-b,+c$, respectively.  Give $E_\ell$ and $P_\ell$ type
$\E$, give $I_\ell$ type $\I$, and assign $C_\ell$ type $\E$ for
definiteness.  Leave all remaining vertices isolated.

Every vertex has refractory period $\rho$ and an initially active age.
Initialize $E_\ell$ and $I_\ell$ at threshold and $C_\ell$ at $x$.  The
first two sources fire at time zero.  Targets have threshold $\theta$, reset
$v_R$, and zero drift.  Give $P_\ell$ a threshold $\theta_P$, reset
$r_P<\theta_P$, no incoming edges, and drift
\[
 b_P(z,\alpha)=\lambda(z-r_P),
 \qquad \lambda>0.
\]
Starting it at
\[
 z_P^0=r_P+(\theta_P-r_P)e^{-\lambda T_P}
\]
makes its pre-firing voltage
$z_P(t)=r_P+(\theta_P-r_P)e^{\lambda(t-T_P)}$.  Thus it is strictly below
threshold for $t<T_P$, fires exactly once at $T_P$ in both
regularizations, resets, and remains at $r_P$.

Put $y=x+a-b$.  The first signed batch makes each target fire once under
E-first and not under I-first by
Lemma~\ref{lem:single-target-order-paper}.  Under E-first the target is
clamped while the inhibitory tail passes and is later released at $v_R$;
under I-first it ends the batch at $y$.  For all large $N$,
\[
 4\varepsilon_N<\min\{\rho/2,T_P-\rho,t_*-T_P\}.
\]
The first batch has therefore ended before $\rho/2$ and the initial source
counts agree, so the first gap is $n_N/N\to1/4$.

The E-first target is released before the probe pulse arrives.  The probe
delivers the same positive increment $c$ in both systems, and
\eqref{eq:probe-parameters} gives
\[
 y+c<\theta<v_R+c.
\]
Thus the E-first target fires a second time during the probe while the
I-first target stays subthreshold.  No later input or target drift exists,
so target counts at $t_*$ are two and zero.  The three source counts agree,
and the second gap is $2n_N/N\to1/2$.  Every target has in-degree three,
every source has out-degree one, and the sign of each edge is determined by
its source type.

\printbibliography

\end{document}